\documentclass[11pt]{amsart}

\usepackage[all]{xy}
\usepackage{fullpage}
\usepackage{latexsym}
\usepackage{amsmath}
\usepackage{amsfonts}
\usepackage{amssymb}
\usepackage{amsthm}
\usepackage{eucal}
\usepackage{enumerate,yfonts}
\usepackage{mathrsfs}
\usepackage{graphicx}
\usepackage{graphics}
\usepackage{epstopdf}
\usepackage{amscd}
\usepackage{bbm}
\usepackage{hyperref}
\usepackage{url}
\usepackage{color}
\usepackage{bbm}
\usepackage{cancel}
\usepackage{enumerate}
\usepackage{amsmath,amsthm}
\usepackage{amssymb}
\usepackage{epsfig}
\usepackage{pstricks}
\usepackage{xy}
\usepackage{xypic}
\usepackage{epigraph}

\newtheorem{thm}{Theorem}[section]
\newtheorem{corollary}[thm]{Corollary}
\newtheorem{lemma}[thm]{Lemma}

\newtheorem{proposition}[thm]{Proposition}

\newtheorem{thm-dfn}[thm]{Theorem-Definition}

\theoremstyle{definition}
\newtheorem{definition}[thm]{Definition}

\numberwithin{equation}{section}

\theoremstyle{remark}
\newtheorem{remark}{Remark}[section]
\newtheorem{example}[remark]{Example}

\newcommand{\fg}{{\mathfrak g}}

\newcommand{\bC}{{\mathbb C}}

\newcommand{\mD}{\mathcal{D}}

\newcommand{\mF}{\mathcal{F}}
\newcommand{\mQ}{\mathcal{Q}}

\newcommand{\mO}{\mathcal{O}}

\newcommand{\mG}{\mathcal{G}}

\newcommand{\on}{\operatorname}

\newcommand{\ra}{\rightarrow}

\newcommand{\is}{\simeq}

\newcommand{\Loc}{\on{Loc}}
\newcommand{\Bun}{\on{Bun}}

\newcommand{\quash}[1]{}  
\newcommand{\nc}{\newcommand}

\newcommand{\bbR}{{\mathbb R}}

\newcommand{\calC}{{\mathcal C}}

\newcommand{\calF}{{\mathcal F}}

\newcommand{\calH}{{\mathcal H}}

\newcommand{\calQ}{{\mathcal Q}}

\newcommand{\calS}{{\mathcal S}}

\newcommand{\calY}{{\mathcal Y}}

\nc{\al}{{\alpha}} \nc{\be}{{\beta}} \nc{\ga}{{\gamma}}
\nc{\ve}{{\varepsilon}} \nc{\Ga}{{\Gamma}} 
\nc{\La}{{\Lambda}}

\nc{\ad }{{\on{ad }}}

\nc{\aff}{{\on{aff}}} \nc{\Aff}{{\mathbf{Aff}}}

\nc{\der}{{\on{der}}}

\nc{\diag}{{\on{diag}}}

\nc{\Fl}{{\calF\ell}}

\nc{\Hg}{{\on{Higgs}}}

\newcommand{\id}{{\on{id}}}
\nc{\Id}{{\on{Id}}}

\nc{\Ind}{{\on{Ind}}}

\nc{\Op}{{\on{Op}}}

\nc{\res}{{\on{res}}}

\nc{\tr}{{\on{tr}}}

\newcommand{\GL}{{\on{GL}}}
\nc{\GSp}{{\on{GSp}}} \nc{\GU}{{\on{GU}}} \nc{\SL}{{\on{SL}}}
\nc{\SU}{{\on{SU}}} \nc{\SO}{{\on{SO}}}

\nc{\nh}{{\Loc_{J^p}(\tau')}}
\nc{\bnh}{{\Loc_{\breve J^p}(\tau')}}

\nc{\bU}{{\overline{U}}} \nc{\IC}{{\on{IC}}}

\newcommand{\beqn}{\begin{equation*}}
\newcommand{\eeqn}{\end{equation*}}

\newcommand{\beq}{\begin{equation}}
\newcommand{\eeq}{\end{equation}}

\quash{

\setlength{\parskip}{2ex}
\setlength{\oddsidemargin}{0in}
\setlength{\evensidemargin}{0in}
\setlength{\textwidth}{6.5in}
\setlength{\topmargin}{-0.15in}
\setlength{\textheight}{8.6in}

\topmargin-0.5cm \textheight22cm \oddsidemargin1.2cm \textwidth14cm}
\begin{document}
\title{Drinfeld-Gaitsgory functor and Matsuki duality}
\author{}
        \address{}
        \email{}
         \author{Tsao-Hsien Chen
     }

        \address{School of Mathematics, University of Minnesota, Twin Cities, 206 Church St. S.E., Minneapolis}
         \email{chenth@umn.edu}
\thanks{}
\thanks{}

\maketitle     
\begin{abstract}
Let $G$
be a connected complex reductive group and let 
$K$ be a symmetric subgroup of $G$.
We prove a formula for the Drinfeld-Gaitsgory functor for
the dg-category $\on{Shv}(K\backslash X)$ of 
$K$-equivariant sheaves on the 
flag manifold $X$ of $G$
in terms of the Matsuki duality functor \cite{MUV}.
As byproducts, we obtain a description of the Serre functor for 
$\on{Shv}(K\backslash X)$, generalizing the 
one in \cite{BBM} in the case of  category $\mO$,
and a formula for 
the Deligne-Lusztig duality for $(\fg,K)$-modules.


\end{abstract}
\section{Introduction}
\subsection{}
For any compactly generated  dg-category $\calC$, Drinfeld-Gaitsgory \cite{G}  
introduced a certain
canonical endofunctor
\[\on{Ps-Id}_\calC:\calC\to\calC\]
to be called the Drinfeld-Gaitsgory functor (in \emph{loc. cit.} they called it the pseudo-identity functor)
with many remarkable properties.
Let $G$ be a connected complex reductive group and let 
$K$ be a symmetric subgroup of $G$.
In this paper, we study the Drinfeld-Gaitsgory functor for 
the dg-category $\calC=\on{Shv}(K\backslash X)_\lambda$ of
$\lambda$-twisted
$K$-equivariant $\bC$-constructible sheaves on the 
flag manifold $X$ of $G$.
Here $\lambda$ is a character of Lie algebra of the abstract Cartan group of $G$. In this geometric setting the Drinfeld-Gaitsgory functor 
\[\on{Ps-Id}_{\on{Shv}(K\backslash X)_\lambda}:\on{Shv}(K\backslash X)_\lambda\to \on{Shv}(K\backslash X)_\lambda\] 
has a concrete description: it
is the functor given by the kernel 
\[(\Delta_{K\backslash X})_!(\bC_{K\backslash X})\in\on{Shv}(K\backslash X\times K\backslash X)_{-\lambda,\lambda}\]
where $\bC_{K\backslash X}$ is the constant sheaf 
on the quotient stack $K\backslash X$ and
$\Delta_{K\backslash X}:K\backslash X\to K\backslash X\times K\backslash X$
is the diagonal map (see Section \ref{DG functor}).

The main result of this paper is 
a formula for the Drinfeld-Gaitsgory functor for $\on{Shv}(K\backslash X)_\lambda$ in terms of the so called Matsuki functors in \cite{MUV}.
To explain the formula, we first recall some notations and results in Lie theory.
Let $G_\bbR$ be the real form of $G$ corresponding to 
$K$ under the 
Cartan bijection.
We write $K_\bbR=K\cap G_\bbR$.
The groups $K$ and $G_\bbR$ act naturally on $X$ with finitely many orbits and the celebrated Matsuki correspondence 
says that there is a natural bijection between the 
$K$-orbits and $G_\bbR$-orbits:
\beq\label{Matsuki bijection}
|K\backslash X|\is|G_\bbR\backslash X|.
\eeq
We have the following functors, to be called the Matsuki functors:
\[\Upsilon^{K\ra G_\bbR}_!:=\on{Av}_!^{G_\bbR/K_\bbR}\circ\on{oblv}_{K/K_\bbR}:\on{Shv}(K\backslash X)_{\lambda}\ra \on{Shv}(G_\bbR\backslash X)_{\lambda}\]
\[\Upsilon^{K\ra G_\bbR}_*:=\on{Av}_*^{G_\bbR/K_\bbR}\circ\on{oblv}_{K/K_\bbR}:\on{Shv}(K\backslash X)_{\lambda}\ra \on{Shv}(G_\bbR\backslash X)_{\lambda}\]
\[\Upsilon^{G_\bbR\ra K}_!:=\on{Av}_!^{K/K_\bbR}\circ\on{oblv}_{G_\bbR/K_\bbR}:\on{Shv}(G_\bbR\backslash X)_{\lambda}\ra \on{Shv}(K\backslash X)_{\lambda}\]
\[\Upsilon^{G_\bbR\ra K}_*:=\on{Av}_*^{K/K_\bbR}\circ\on{oblv}_{G_\bbR/K_\bbR}:\on{Shv}(G_\bbR\backslash X)_{\lambda}\ra \on{Shv}(K\backslash X)_{\lambda}.\]
Here for any Lie group $H$ and a Lie subgroup $H'\subset H$, the functor 
$\on{oblv}_{H/H'}:\on{Shv}(H\backslash X)_{\lambda}\to\on{Shv}(H'\backslash X)_{\lambda}$ is the natural forgetful functor
and $\on{Av}_*^{H/H'}$ (resp. $\on{Av}_!^{H/H'}$) is its right (resp. left) adjoint.
In \cite{MUV}, the authors updated the Matsuki correspondence~\eqref{Matsuki bijection} to the level of sheaves:
\beq\label{M functor}
\text{
The Matsuki functor 
$\Upsilon^{K\ra G_\bbR}_*$ (resp. $\Upsilon_!^{G_\bbR\ra K}$)
is an equivalence with inverse given by}
\eeq
\[
\text{$\Upsilon_!^{G_\bbR\ra K}$ (resp. $\Upsilon_*^{G_\bbR\ra K}$).}\]

The following theorem is the main result of the note.
Write $\on{Ps-Id}_{K\backslash X}=\on{Ps-Id}_{\on{Shv}(K\backslash X)_\lambda}$ for the Drinfeld-Gaitsgory functor for 
$\on{Shv}(K\backslash X)_\lambda$.
For simplicity, we assume $K$ is connected.
\begin{thm}[Theorem \ref{main 1 twisted}]\label{formula 1}
There is a canonical isomorphism of functors $\on{Shv}(K\backslash X)_\lambda\ra \on{Shv}(K\backslash X)_\lambda$:
\[\on{Ps-Id}_{K\backslash X}\is\Upsilon_!^{G_\bbR\ra K}\circ\Upsilon^{K\ra G_\bbR}_![-\dim_\bbR(X)+\dim_\bbR(K_\bbR)].\]
\end{thm}

It was shown in \cite[Theorem 2.1.5]{GY} that 
$\on{Ps-Id}_{K\backslash X}$ is 
isomorphic to the inverse 
of the 
Serre functor $\on{Se}_{K\backslash X}:\on{Shv}(K\backslash X)_\lambda\ra \on{Shv}(K\backslash X)_\lambda$.
Thus, in view of~\eqref{M functor}, we obtain 
\begin{corollary} [Theorem \ref{main 2 twisted}]\label{formula 2}
There is a canonical isomorphism of functors $\on{Shv}(K\backslash X)_\lambda\ra \on{Shv}(K\backslash X)_\lambda$:
\[\on{Se}_{K\backslash X}\is\Upsilon_*^{G_\bbR\ra K}\circ\Upsilon^{K\ra G_\bbR}_*[\dim_\bbR(X)-\dim_\bbR(K_\bbR)].\]
\end{corollary}

In fact, Theorem \ref{main 1 twisted} and Theorem \ref{main 2 twisted}
are in a more general setting 
of \emph{Matsuki datum}, see Section \ref{Matsuki duality}.

\begin{remark}
If $K$ (equivalently $K_\bbR$ or $G_\bbR$) is disconnected, one need to modify the formulas above for $\on{Ps-Id}_{K\backslash X}$ and
$\on{Se}_{K\backslash X}$ by
inserting the functor $(-)\otimes\ell_{K\backslash X}$
of tensoring with a certain rank one local system on 
$K\backslash X$ coming from the orientation sheaf $\on{or}_{K_\bbR\backslash X}$ of $K_\bbR\backslash X$, see Remark \ref{twisting}.
When $K$ is connected, $\on{or}_{K_\bbR\backslash X}$ is trivial 
and we do not need such modification, see Remark \ref{modification}.

\end{remark}

\begin{remark}
The formulas above for the Drinfeld-Gaitsgory functor and the Serre functor 
generalize the one in
\cite[Theorem 3.4.2]{CGY} and \cite[Proposition 2.5]{BBM}
to the setting of real groups.
Indeed, let $N$ be the
unipotent radical of a Borel subgroup of $G$
and let $N^-$ be its opposite unipotent radical.
In~\emph{loc. cit.} the authors showed that the 
Drinfeld-Gaitsgory functor $\on{Ps-Id}_{N\backslash X}$
and the Serre functor $\on{Se}_{N\backslash X}$ for the category $\on{Shv}(N\backslash X)$
of sheaves on the stack $N\backslash X$ (equivalently, the category $\mO$) are isomorphic to
\[\on{Ps-Id}_{N\backslash X}\is\Upsilon_!^{N^-\to N}\circ\Upsilon_!^{N\to N^-}[-2\dim_\bC(X)]\ \ \ \ \on{Se}_{N\backslash X}\is\Upsilon_*^{N^-\to N}\circ\Upsilon_*^{N\to N^-}[2\dim_\bC(X)]\] 
where the functors $\Upsilon_!^{N\to N^-},\Upsilon_*^{N\to N^-}:\on{Shv}(N\backslash X)\to\on{Shv}(N^{-}\backslash X)$
are the so called long-intertwining functors.
Our results 
suggest that the Matsuki functors are ``real group'' analog
of the long-intertwining functors.
\end{remark}

\subsection{A formula for the Deligne-Lusztig duality for $(\fg,K)$-modules }
In \cite{GY}, Gaitsgory-Yom Din 
introduced an analog of the \emph{Deligne-Lusztig duality} for $(\fg,K)$-modules
as the composition of the 
canonical duality $\mathbb D^{\text{can}}$ and the contragrediant duality $\mathbb D^{\text{contr}}$:
\beq
\mathbb D^{\text{can}}\circ\mathbb D^{\text{contr}}:\fg\on{-mod}^K_\chi\to \fg\on{-mod}^K_\chi.
\eeq
The main result in~\emph{loc.cit.} says that 
the Drinfeld-Gaitsgory functor
$\on{Ps-Id}_{\fg\on{-mod}^K_\chi}$ for $(\fg,K)$-modules
is isomorphic to the Deligne-Lusztig functor
\beq\label{DL functor}
\on{Ps-Id}_{\fg\on{-mod}^K_\chi}\is\mathbb D^{\text{can}}\circ\mathbb D^{\text{contr}}
\eeq
Since the category $\fg\on{-mod}^K_\chi$ is equivalent to $\on{Shv}(K\backslash X)_{\lambda}$ 
via the localization theorem and the Riemann-Hilbert correspondence (or rather the dg-category version, see, e.g., \cite{P}), Theorem \ref{formula 1} gives 
a formula for the Deligne-Lusztig functor 
for $(\fg,K)$-modules in terms of Matsuki functors:

\begin{corollary}\label{formula for DL}
We have an isomorphism of functors on $\fg\on{-mod}^K_{\chi}\is\on{Shv}(K\backslash X)_\lambda$:
\[\mathbb D^{\on{can}}_{}\circ\mathbb D^{\on{contr}}\is\Upsilon_!^{G_\bbR\ra K}\circ\Upsilon^{K\ra G_\bbR}_![-\dim_\bbR(X)+\dim_\bbR(K_\bbR)]\]
\end{corollary}

\begin{remark}
It is interesting to note that the formula above for the Deligne-Lusztig duality 
for $(\fg,K)$-modules
is not algebraic but only \emph{real analytic}; as it involves the Matsuki functors 
going from $\on{Shv}(K\backslash X)_\lambda$ to $\on{Shv}(G_\bbR\backslash X)_\lambda$.
A similar phenomenon also appeared in the work of Barlet-Kashiwara \cite{BK}, where the authors
described the functor on $\on{Shv}(K\backslash X)_\lambda$ which corresponds
to the contragradient duality $\mathbb D^{\text{contr}}$ for $(\fg,K)$-modules and they 
also observed that the functor is  not algebraic but only real analytic. 
It will be interesting to compare the formula obtained here and the one in~\emph{loc. cit.}.
\end{remark}

\begin{remark}
It would be nice to have an analog of Corollary \ref{formula for DL}
for $G_\bbR$-representations.
Note that we have similar formulas in the 
real setting: 
the Drinfeld-Gaitsgory and 
the Serre functor for  $\on{Shv}(G_\bbR\backslash X)_\lambda$
are isomorphic to:
\[\on{Ps-Id}_{G_\bbR\backslash X}\is\Upsilon_!^{K\ra G_\bbR}\circ\Upsilon^{G_\bbR\ra K}_![-\dim_\bbR(X)+\dim_\bbR(K_\bbR)]\]
\[\on{Se}_{G_\bbR\backslash X}\is\Upsilon_*^{K\ra G_\bbR}\circ\Upsilon^{G_\bbR\ra K}_*[\dim_\bbR(X)-\dim_\bbR(K_\bbR)].
\]
In \cite{KSd}, Kashiwara-Schmid constructed a functor from 
$\on{Shv}(G_\bbR\backslash X)_\lambda$ to (certain) derived category of 
$G_\bbR$-representations, known as Kashiwara-Schmid
localization. Similar to the case of $(\fg,K)$-modules, we expect that the Kashiwara-Schmid
localization intertwines the 
Drinfeld-Gaitsgory functor 
$\on{Ps-Id}_{G_\bbR\backslash X}$
with 
the composition $\mathbb D^{\on{coh}}\circ\mathbb D^{\on{contr}}$
where 
$\mathbb D^{\on{contr}}$ and $\mathbb D^{\on{coh}}$ are the 
the contragredient duality 
and (some version) of cohomological duality for 
$G_\bbR$-representations respectively.
At the level of virtual $G_\bbR$-representations (up to infinitesimal equivalence),
this follows from the 
fact that the Matsuki functor intertwines the Kashiwara-Schmid
localization with the Beilinson-Bernstein localization  (see, e.g., \cite[Section 2]{SV}).
\end{remark}
\subsection{Further directions}

\subsubsection{Theory of tilting perverse sheaves}
In \cite{BBM}, the authors use 
the theory of tilting perverse sheaves 
on $N\backslash X$ to derive the formula for Serre functor 
for category $\mO$. It would be nice if one can do the same for $(g,K)$-modules:
study tilting perverse sheaves on $K\backslash X$ or $G_\bbR\backslash X$ and explore its connection with the Drinfeld-Gaitsgory functor and the Serre functor.

\subsubsection{Affine Matsuki duality}
In \cite{CN}, Nadler and the author prove a affine version of the Matsuki duality, to be called the affine Matsuki duality, which relates the category 
of sheaves on the moduli stack $\Bun_{G_\bbR}$ of
real bundles over real projective line with the category 
spherical sheaves on the loop space $X((t))$ of the symmetric variety $X=G/K$.
The results of the paper suggest that there is a close relationship between 
the Drinfeld-Gaitsgory functor for those categories 
and the affine Matsuki duality. Note that $\Bun_{G_\bbR}$
and $X((t))$ are the main players for 
the geometric Langlands for real groups (proposed by Ben-Zvi and Nadler)
and derived Satake equivalence for symmetric varieties (proposed by 
Ben-Zvi, Venkatesh, and Sakellaridis), and we expect that the
affine Matsuki duality (and its relation with the Drinfeld-Gaitsgory functor)
would help us to understand some ``duality'' phenomena in both subjects.

\subsection{Acknowledgements} 
I am grateful to D. Gaitsgory and A. Yom Din 
 for the collaboration \cite{CGY}.
The research is supported by NSF grant DMS-2001257 and the S. S. Chern Foundation.

\section{constructible sheaves on a semi-analytic stack}
\subsection{}
We will be working will $\bC$-linear dg-categories (see \cite[Section 0.6]{DG} for a concise summary of the theory of DG categories). Unless specified otherwise, all dg-categories will be assume cocomplete, i.e., contains all small colimits,
and all functors between dg-categories will be assumed continous, i.e., preserves all small colimits.

Recall that a subset $S$ of a 
real analytic manifold $M$ 
is called semi-analytic if
any point $s\in S$ has a open neighbourhood $U$ 
such that the intersection $S\cap U$ is a finite union of sets of the form
\[\{s\in U|f_1(s)=\cdot\cdot\cdot=f_r(s)=0, g_1(s)>0,..., g_l(s)>0\},\]
where the $f_i$ and $g_j$ are real analytic functions on 
$U$.
A map $f:S\to S'$ between two semi-analytic sets is called semi-analytic if
it is continuous and its graph is a semi-analytic set.

We collect some facts about constructible sheaves on a semi-analytic stack following \cite[Appendix C]{AGKRRV1}.\footnote{In \emph{loc. cit.}, the authors work with constructible sheaves on a  
complex semi-analytic stack but the discussion works for real semi-analytic stacks.}
For any semi-analytic set $S$, 
we define $\on{Shv}(S)=\on{Ind}(\on{Shv}(S)^{\on{constr}})$ to be the ind-completion
of the 
 bounded dg-category $\on{Shv}(S)^{\on{constr}}$ of $\bC$-constructible sheaves on 
$S$. For any semi-analytic stack $\calS$ 
we define $\on{Shv}(\calS):=\underset{S}{\on{lim}}\ \on{Shv}(S)$
where the index category is that of semi-analytic sets equipped with a semi-analytic map
to $\calS$, and the transition functors are given by $!$-pullback.
Since we are in the constructible context  $!$-pullback admits left adjoint, given by 
$!$-pushforward, and it follows that 
$\on{Shv}(\calS)=\underset{S}{\on{colim}}\on{Shv}(S)$. In particular,
$\on{Shv}(\calS)$ is compactly generated. 
We let $\on{Shv}(\calS)^c$ be the full subcategory of compact objects
and $\on{Shv}(\calS)^{\text{constr}}\subset\on{Shv}(\calS)$ be the full subcategory
of objects that pullback to an object of $\on{Shv}(S)^{\on{constr}}$ for any $S$ mapping to $\calS$.
Note that we have natural inclusion $\on{Shv}(\calS)^c\subset\on{Shv}(\calS)^{\text{constr}}$
but the inclusion is in general
not an equality. For example, the constant sheaf $\bC_{\calS}\in\on{Shv}(\calS)^{\text{constr}}$  
for the classifying stack $\calS=B(\GL_1(\bC))$ is not compact.

Let $f:\calS\to\calS'$ be a map between semi-analytic stacks.
We have the usual six functor formalism
$f_*$, $f^!$, $f_*$, $f_!$, $\otimes$, $\underline{\on{Hom}}$.
For any $\mF_1,\mF_2\in\on{Shv}(\calS)$ we denote by
$\mF_1\otimes^!\mF_2:=\Delta^!_\calS(\mF_1\boxtimes\mF_2)$,
where $\Delta_\calS:\calS\to\calS\times\calS$ be the diagonal map.

\subsection{Renormalized direct image}
 The usual direct image functor 
$f_*:\on{Shv}(\calS)\to\on{Shv}(\calS')$ is in general not continuous
(see, e.g., \cite[Example 7.1.4]{DG}). 
We define the \emph{renormalized} direct image
\beq\label{ren}
f_{*,\text{ren}}:\on{Shv}(\calS)\to\on{Shv}(\calS')
\eeq
to be the unique colimit-preserving functor such that 
$f_{*,\text{ren}}=f_*$ on $\on{Shv}(\calS)^c$.

Following \cite[Appendix C]{AGKRRV1},  a semi-analytic stack $\calS$ is called 
\emph{duality-adapted} if 
the Verdier duality functor 
\[\mathbb D:(\on{Shv}(\calS)^{\text{constr}})^{\on{op}}\is\on{Shv}(\calS)^{\text{constr}}\]
sends $\on{Shv}(\calS)^{c}$ to $\on{Shv}(\calS)^{c}$. 
In this paper we are mainly interested in 
semi-analytic stacks $\calS$ of the form 
$\calS=H\backslash S$ where $H$ is Lie group acting semi-analytically on a semi-analytic set 
$S$. According to  \cite[Theorem C.2.6]{AGKRRV1}, any such stack is duality-adapted.
By \cite[C.3.6]{AGKRRV1}, given a map 
$f:\calS\to\calS'$, where $\calS$ is duality-adapted, the functor 
$f_{*,\on{ren}}$ satisfies the projection formula 
\[f_{*,\on{ren}}(\mF)\otimes^!\mF'\is f_{*,\on{ren}}(\mF\otimes^!f^!(\mF'))\]
and base change: for a Cartesian diagram
\[\xymatrix{\calS_1\ar[r]^{g_1}\ar[d]_{f_1}&\calS_2\ar[d]^{f_2}\\
\calS'_1\ar[r]_{g_2}&\calS_2'}\]
we have 
\[g_2^!\circ f_{2*,\on{ren}}\is f_{1*,\on{ren}}\circ g_1^!.\]
\begin{remark}
Note that 
the usual direct image functor 
$f_*$ for a map $f:\calS\to\calS'$ between semi-analytic stacks might not satisfy the projection formula or base change 
(see \cite[Example 7.7.6]{DG}) and this is one of the reason 
to introduce the renormalized version $f_{*,\on{ren}}$.
However, according to \cite[Corollary 10.2.5]{DG}, if the map $f$ is schematic (or more generally, the map is ``safe'' as defined in \emph{loc. cit.}) then 
we have $f_*=f_{*,\on{ren}}$ and hence $f_*$ satisfies 
the projection formula or base change.

\end{remark}

\subsection{Orientation sheaf}
Let $f:\calS\to\calS'$ be a map between semi-analytic stacks.
We will write $d_f=\dim_\bbR\calS-\dim_\bbR\calS'$ for the relative dimension of $f$.
The complex $\omega_f:=f^!\bC_{\calS'}$ is called the (relative) dualizing complex for $f$.
If $f$ is smooth then we have $\omega_f\is\on{or}_f[d_f]$ where 
$\on{or}_f:=\calH^{-d_f}(\omega_f)$ is a local system on 
$\calS$ called the (relative) orientation sheaf of $f$. 
Note that we have $\on{or}_f\otimes\on{or}_f\is\bC_\calS$.
Moreover, have a canonical isomorphism of functors
\beq\label{smooth pull back}
f^*\otimes\omega_f\is f^*\otimes\on{or}_f[d_f]\is f^!.
\eeq
In the case when $S'=\on{pt}$ we will write 
$d_\calS=d_f$,
$\omega_\calS=\omega_f$, $\on{or}_\calS=\on{or}_f$, etc. 
\subsection{Averaging functors}
Let $X$ be a real analytic manifold acted on by 
a Lie group $H$. Let $U\subset H$ be Lie subgroup 
and consider the quotient map
$f:U\backslash X\to H\backslash X$.
We have the forgetful functor
\[\on{oblv}_{H/U}=f^*:\on{Shv}(H\backslash X)\to\on{Shv}(U\backslash X)\]
with right adjoint 
\[\on{Av}_*^{H/U}=f_*:\on{Shv}(U\backslash X)\to\on{Shv}(H\backslash X)\]
(note that $f_*=f_{*,\on{ren}}$ as $f$ is schematic) and left adjoint
\[\on{Av}_!^{H/U}=f_!\circ((-)\otimes\on{or}_f)[d_f]:\on{Shv}(U\backslash X)\to\on{Shv}(H\backslash X)\]
here $(-)\otimes\on{or}_f:\on{Shv}(U\backslash X)\to\on{Shv}(U\backslash X)$ is the functor of tensoring with the 
orientation sheaf $\on{or}_f$.

\section{The Drinfeld-Gaitsgory functor}\label{DG functor}
\subsection{}
We recall the definition of Drinfeld-Gaitsgory functor 
in the setting of constructible sheaves following \cite{AGKRRV2}.
Let $\calS$ be semi-analytic stack. 
An object $\mQ\in\on{Shv}(\calS\times\calS)$ defines a functor
\[F_\mQ:\on{Shv}(\calS)\to\on{Shv}(\calS),\ \ \mF\to (p_2)_{*,\on{ren}}(\mQ\otimes^!(p_1)^!(\mF))\] 
where
$\calS\stackrel{p_1}\longleftarrow\calS\times
\calS\stackrel{p_2}\longrightarrow\calS$ are the projection maps.
We shall call $F_\calQ$ the functor given by the kernel $\mQ$.
Consider the diagonal embedding 
$\Delta_\calS:\calS\to\calS\times\calS$.
It is straightforward to check that the identity functor 
\[\on{Id}_\calS:\on{Shv}(\calS)\to\on{Shv}(\calS)\]
is given by the kernel $(\Delta_\calS)_*(\omega_\calS)$.
The \emph{Drinfeld-Gaitsgory functor}:
\[\on{Ps-Id}_{\calS}:\on{Shv}(\calS)\to\on{Shv}(\calS)\]
is by definition the functor given by the 
kernel $(\Delta_\calS)_!(\bC_\calS)$ where $\bC_\calS$ is the constant sheaf on $\calS$.

\section{Matsuki duality}\label{Matsuki duality}
\subsection{}
We recall the definition of ``Matsuki datum''
introduced in \cite[Section 5]{MUV}.
\begin{definition}
Matsuki datum $(G,X,H^+,H^-)$ consists of 
a linear Lie group $G$ acting on a 
real analytic manifold $X$ and two subgroups $H$ and $H^-$ of 
$G$ such that there is a Bott-Morse function 
$f:X\to\mathbb R$ invariant under $H^+\cap H^-=U$ and 
a $U$-invariant metric on $X$, satisfying the following:
\begin{enumerate}
\item The gradient flow $\Phi$ of $f$ on $X$ preserves the 
$H^\pm$-orbits and $U$ has finitely many orbits in the fixed point $X^\Phi$ set
of the flow.
\item For any $U$-orbit $\mO$ in $X^\Phi$ denote 
\[\mO^\pm=\{x\in X| \Phi_t(x)\in\mO\text{\ as t goes to\ }\pm\infty\}.\]
Then $\mO^\pm$ are single $H^\pm$-orbits and the correspondence 
$\mO^+\leftrightarrow\mO^-$ is a bijection between $H^+$-and $H^-$-orbits on $X$.
\item The $H^+$-and $H^-$-orbits on $X$ are transversal.
\item The quotient $H^\pm/U$ is contractible.
\end{enumerate} 
\end{definition}
\begin{remark}
Note that if $(G,X,H^+,H^-)$ is a Matsuki datum then so is 
$(G,X,H^-,H^+)$.
\end{remark}
\begin{example}
Let $G$ be a connected complex reductive group 
and let $X$ be the flag manifold of $G$.
(1) Let $P=MN$ be a parabolic subgroup with 
unipotent radical $N$ and let $P^{-}=MN^-$ be its opposite.
Then $(G,X,MN,MN^-)$ is a Matsuki datum.
(2)
Let $G_\bbR$ be a real form of $G$ with a maximal compact subgroup
$K_\bbR\subset G_\bbR$ and let $K\subset G$ be its complexification.
According to \cite{MUV}, $(G,X,G_\bbR,K)$ (resp. $(G,X,K,G_\bbR)$)
is a Matsuki datum.
\end{example}
\subsection{}
Given a Matsuki datum $(G,X,H^+,H^-)$, one can consider the functors
\[\Upsilon_!^{H^+\to H^-}=\on{Av}_!^{H^-/U}\circ\on{oblv}_{H^+/U}:\on{Shv}(H^+\backslash X)_{}\ra \on{Shv}(H^-\backslash X)_{}\]
\[\Upsilon_*^{H^-\to H^+}=\on{Av}_*^{H^+/U}\circ\on{oblv}_{H^-/U}:\on{Shv}(H^-\backslash X)_{}\ra \on{Shv}(H^+\backslash X)_{}.\]
\begin{thm}\label{Matsuki}
The functor 
$\Upsilon_!^{H^+\ra H^-}$ 
is an equivalence with inverse given by  $\Upsilon^{H^-\ra H^+}_*$.
\end{thm}
\begin{proof}
In the setting of bounded derived category this is the main result in  
\cite[Theorem 5.3]{MUV}.
In the setting of dg-category, it follows from  the alternative argument in  
\cite[Theorem 5.2]{CY} or \cite[Proposition 1.4.2]{CGY}.
Namely, we would like to show that the canonical adjunctions
$\on{Id}\to\Upsilon_*^{H^-\to H^+}\circ\Upsilon_!^{H^+\to H^-}$
and $\Upsilon_!^{H^+\to H^-}\circ\Upsilon_*^{H^-\to H}\to \on{Id}$
are isomorphisms.
Consider the category $\on{Shv}(U\backslash X)_{\on{sm}}\subset\on{Shv}(U\backslash X)$
consisting of $U$-equivaraint complexes whose cohomology sheaves are
smooth along the trajectories of the Morse flow $\Phi$.
We have $\on{oblv}_{H^\pm/U}(\mF)\in\on{Shv}(U\backslash X)_{\on{sm}}$ and as explained in \emph{loc. cit.} the 
desired claim follows from the fact that, on the category 
$\on{Shv}(U\backslash X)_{\on{sm}}$, 
 the functor $\on{Av}_!^{H^+/U}$ and 
$\on{Av}_*^{H^-/U}$ act as identity on hyperbolic restriction to
the fixed points set $X^\Phi$.\footnote{In \emph{loc. cit.}, we work in the setting 
of ``complex analytic'' Matuski datum where all the varieties are 
complex analytic, but the same argument works equally well in the 
semi-analytic setting.}
\end{proof}

\section{Singular support and transversal property}
\subsection{}
We recall some basic definition and properties of singular support of a complex of sheaves.
The reference is \cite{KSa}.
Let $S$ be a real analytic manifold. 
Given a closed conical semi-analytic subset $\Lambda\subset T^*S$ we denote 
by $\on{Shv}_{\Lambda}(S)^{\on{constr}}\subset\on{Shv}_{}(S)^{\on{constr}}$ the full dg-category of bounded constructible complexes 
$\mF$
with singular support $\on{sing}(\mF)\subset\Lambda$ and 
$\on{Shv}_{\Lambda}(S)\subset\on{Shv}_{}(S)^{}$ be the ind-completion of $\on{Shv}_{\Lambda}(S)^{\on{constr}}$.

Let $\calS$ be a smooth semi-analytic stack, and let $T^*\calS$ be the cotangent bundle.
Recall that for any smooth map $f:S\to\calS$ where $S$ is a real analytic manifold, we have a correspondence
\[T^*S\stackrel{df}\longleftarrow T^*\calS\times_\calS S\stackrel{pr}\longrightarrow T^*\calS.\]
For any conical closed semi-analytic sub-stack $\Lambda\subset T^*\calS$, 
and a smooth map $f:S\to\calS'$ as above 
we denote by
\[\Lambda_S:=df(pr^{-1}(\Lambda))\subset T^*S\]
and 
we define $\on{Shv}_{\Lambda}(\calS)^{\on{constr}}\subset\on{Shv}_{}(\calS)^{\on{constr}}$ to be the full dg-category of bounded constructible complexes 
$\mF$ such that 
\[\on{sing}(f^*\mF)\subset\Lambda_S.\]
We let $\on{Shv}_{\Lambda}(\calS)\subset\on{Shv}_{}(S)^{}$ be the ind-completion of $\on{Shv}_{\Lambda}(\calS)^{\on{constr}}$.

\begin{definition}
Let $f:\calS\to\calS'$ be a map between smooth semi-analytic stacks.
We say that a complex $\mF\in\on{Shv}(\calS)$ is \emph{transversal} with respect to $f$ if 
there is a closed conical semi-analytic subset $\Lambda\subset T^*\calS$ such that 
(1)
$\mF\in\on{Shv}_\Lambda(\calS)$ and (2) $\Lambda\times_{T^*\calS}(T^*\calS'\times_{\calS'}\calS)=\text{zero section }T^*_{\calS}\calS$ of $T^*\calS$.

\end{definition}

\begin{example}\label{example transversal}
Let $X$ be a compact real analytic manifold acted on by two Lie groups $H^+$ and $H^-$
and let $U\subset H^+\cap H^-$ be a Lie subgroup.
Consider the natural map $f^\pm:\calS=U\backslash X\to \calS'=H^\pm\backslash X$.
We claim that if $H^-$-and $H^+$-orbits on $X$ are transversal then, for any 
$\mF\in\on{Shv}(H^-\backslash X)$, its image
$\on{oblv}_{H^-/U}(\mF)\in\on{Shv}(U\backslash X)$ is transversal with respect to $f^+$.
Indeed, let $\Lambda^\pm\subset T^*(U\backslash X)$ be the image of 
$df^\pm:T^*(H^\pm\backslash X)\times_{H^\pm\backslash X}U\backslash X\to T^*(U\backslash X)$.
Then $\Lambda^\pm= T^*_{H^\pm}(X)/U\subset  T^*_{U}(X)/U= T^*(U\backslash X)$ 
(here $T^*_{U}(X)$, etc, stands for the conormal bundle of 
$U$-orbtis on $X$)
and we have (1) 
$\mF\in\on{Shv}_{\Lambda^-}(U\backslash X)$
and (2) $\Lambda^-\times_{T^*\calS}(T^*\calS'\times_{\calS'}\calS)=\Lambda^-\cap\Lambda^+=T^*_{H^-}(X)\cap T^*_{H^+}(X)/U$
which is the zero section $T^*_{U\backslash X}(U\backslash X)$ as 
$T^*_{H^-}(X)\cap T^*_{H^+}(X)=T^*_XX$.

\end{example}

Let $X$, $U$, $H=H^+$, and $f=f^+:U\backslash X\to H\backslash X$ be as in the Example \ref{example transversal}.
Let $\on{or}_{U\backslash X}$ be the orientation sheaf of $U\backslash X$.
We assume $H/U$ is contractible.
Then the map $f:U\backslash X\to H\backslash X$ is smooth with contractible fibers 
there exists a unique rank one local system 
$\ell_{H\backslash X}$ on $H\backslash X$ such that $f^*\ell_{H\backslash X}\is\on{or}_{U\backslash X}$.
Recall the pseudo-identity functor $\on{Ps-Id}_{H\backslash X}:\on{Shv}(H\backslash X)\ra \on{Shv}(H\backslash X)$.

Let $U_c$ be a maximal compact subgroup of $U$.
Note that the quotient $U/U_c$ is contractible.
The following lemma is an analogy of \cite[Theorem 3.2.6]{CGY} in the real analytic setting:
\begin{lemma}\label{lemma 1}
There is a canonical defined morphism of functors $\on{Shv}(U\backslash X)\ra \on{Shv}(H\backslash X)$:
\beq\label{arrow}
(-\otimes\ell_{H\backslash X})\circ\on{Av}_!^{H/U}\ra\on{Ps-Id}_{H\backslash X}\circ\on{Av}_*^{H/U}[d_X-d_{U_c}].
\eeq
The map (\ref{arrow}) is an isomorphism when evaluated on objects
which are transversal with respect the projection 
$U\backslash X\ra H\backslash X$.
\end{lemma}

\begin{proof}
We follow the argument in \emph{loc.cit.} closely. 
All the direct image functors below should be understood as the 
\emph{renormalized} version in Section \ref{ren} and,
for simplicity, we still use the notation $(p_1)_*=(p_1)_{*,\on{ren}}$, etc.
Consider the following diagram
\[\xymatrix{U\backslash X\ar[r]^{\tilde\Delta\ \ \ \ }\ar[d]^f&U\backslash X\times H\backslash X\ar[d]^{f\times\id}\ar[r]^{\ \ \ \ \tilde p_1}&U\backslash X\ar[d]^f\\
H\backslash X\ar[r]^{\Delta\ \ \ \ }&H\backslash X\times H\backslash X\ar[d]^{p_2}\ar[r]^{\ \ \ \ p_1}&H\backslash X\\
&H\backslash X&}\]
Let $\mF\in\on{Shv}_U(X)$. 
We have 
\[(-\otimes\ell_{H\backslash X})\circ\on{Av}_!^{H/U}(\mF)\is \ell_{H\backslash X}\otimes f_!(\mF\otimes\on{or}_f)[d_f]\is f_!(\mF\otimes\on{or}_f\otimes f^*(\ell_{H\backslash X}))[d_f]\is\]
\[\is 
f_!(\mF\otimes\on{or}_f\otimes\on{or}_{U\backslash X})[d_f]\is
(\tilde p_2)_!\circ\tilde\Delta_!(\mF\otimes\on{or}_f\otimes\on{or}_{U\backslash X}))[d_f]\is\]
\[\is(\tilde p_2)_!\circ\tilde\Delta_!\circ\tilde\Delta^*\circ\tilde p_1^*(\mF\otimes\on{or}_f\otimes\on{or}_{U\backslash X}))[d_f]\is\]
\[\is(\tilde p_2)_!(\tilde\Delta_!(\bC_{U\backslash X})\otimes\tilde p_1^*(\mF\otimes\on{or}_f\otimes\on{or}_{U\backslash X}))[d_f]\is\]
\[\is(\tilde p_2)_!(\tilde\Delta_!(\on{or}_f)\otimes\tilde p_1^*(\mF\otimes\on{or}_{U\backslash X}))[d_f].\]
Here $\tilde p_2=p_2\circ (f\times\id)$. 
Note also that 
\[\on{Ps-Id}_{H\backslash X}\circ\on{Av}_*^{H/U}(\mF)\is(p_2)_{*}(\Delta_!(\bC_{H\backslash X})\otimes^!(p_1^!\circ f_*(\mF)))\is(p_2)_{*}(\Delta_!(\bC_{H\backslash X})\otimes^!(f\times\id)_*\circ\tilde p_1^!(\mF)))\is\]
\[\stackrel{}\is(\tilde p_2)_*((f\times\id)^!\circ\Delta_!(\bC_{H\backslash X})\otimes^! (\tilde p_1)^!\mF)\is(\tilde p_2)_*(\tilde\Delta_!\circ(f^!(\bC_{H\backslash X})\otimes^! \tilde p_1^!(\mF))\is\]
\[\is(\tilde p_2)_*(\tilde\Delta_!(\on{or}_f)\otimes^! \tilde p_1^!(\mF))[d_f]\]
Since $\tilde p_1$ is smooth and $X$ is compact, it follows from~\eqref{smooth pull back} and Lemma \ref{p_*=p_!}
below that
\[\tilde p_1^!(\mF)\is \tilde p_1^*(\mF)\otimes\on{or}_{\tilde p_1}[d_{\tilde p_1}]\ \ \ \ \ \ (p_2)_*\is (p_2)_![-d_{U}+d_{U_c}].\]
All together we obtain
\[(-\otimes\ell_{H\backslash X})\circ\on{Av}_!^{H/U}(\mF)\is(\tilde p_2)_!(\tilde\Delta_!(\on{or}_f)\otimes\tilde p_1^*(\mF)\otimes\tilde p_1^*(\on{or}_{U\backslash X})))[d_f]\]
\[\on{Ps-Id}_{H\backslash X}\circ\on{Av}_*^{H/U}(\mF)\is(\tilde p_2)_!(\tilde\Delta_!(\on{or}_f)\otimes^! (\tilde p_1^*(\mF)\otimes\on{or}_{\tilde p_1}))[d_f+d_{\tilde p_1}-d_{U}+d_{U_c}].\]
Write $\calY=U\backslash X\times H\backslash X$,
$\mF_1=\tilde p_1^*(\mF)$, $\mF_2=\tilde\Delta_!(\on{or}_f)$, 
and let $\Delta_\calY:\calY\to\calY\times\calY$ be the diagonal imbedding.
Note that there is a canonical arrow
\beq\label{can arrow}
\mF_2\otimes\mF_1\otimes\Delta_\calY^!(\bC_{\calY\times\calY})\is
\Delta^*(\mF_2\boxtimes\mF_1)\otimes\Delta_\calY^!(\bC_{\calY\times\calY})\to
\Delta^!(\mF_2\boxtimes\mF_1)\is\mF_2\otimes^!\mF_1
\eeq
Since 
\[\Delta_\calY^!(\bC_{\calY\times\calY})\is\on{or}_\calY[-d_\calY]\ \ \ \ \ \ 
\on{or}_{\calY}\is\on{or}_{\tilde p_1}\otimes \tilde p_1^*(\on{or}_{U\backslash X}) \]
\[d_X-d_{U_c}=d_{\calY}-(d_{\tilde p_1}-d_{U}+d_{U_c}),\]
it follows that there is a canonical arrow
\[\mF_2\otimes\mF_1\otimes\tilde p_1^*(\on{or}_{U\backslash X})\is
\mF_2\otimes\mF_1\otimes\Delta_\calY^!(\bC_{\calY\times\calY})\otimes\on{or}_{\tilde p_1}[d_\calY]\to
\mF_2\otimes^!\mF_1\otimes\on{or}_{\tilde p_1}[d_\calY]\]
and the desired arrow in~\eqref{arrow} is given by
\[(-\otimes\ell_{H\backslash X})\circ\on{Av}_!^{H/U}(\mF)\is (\tilde p_2)_!(\mF_2\otimes\mF_1\otimes\tilde p_1^*(\on{or}_{U\backslash X}))[d_f]\stackrel{}\longrightarrow
(\tilde p_2)_!(\mF_2\otimes^!\mF_1\otimes\on{or}_{\tilde p_1})[d_f+d_\calY]\is\]
\[\is
\on{Ps-Id}_{H\backslash X}\circ\on{Av}_*^{H/U}(\mF)[d_X-d_{U_c}].\]

Assume $\mF$ is transversal with respect to $f:U\backslash X\to H\backslash X$.
Then it is straightforward to check that 
$\mF_2\boxtimes\mF_1\in\on{Shv}_{}(\calY\times\calY)$ is \emph{non-characteristic} with respect to
$\Delta_\calY:\calY\to\calY\times\calY$, that is, the restriction map  
$\on{sing}(\mF_2\boxtimes\mF_1)|_\calY \subset T^*(\calY\times\calY)|_\calY\to T^*\calY$
of the differential map of $\Delta_\calY$ to the singular support $\on{sing}(\mF_2\boxtimes\mF_1)$
is an injective map. It follows from \cite[Proposition 5.4.13]{KSa} that 
the canonical arrow in~\eqref{can arrow}, and hence the one in~\eqref{arrow}, is an isomorphism.
The proof is completed.

\end{proof}

\begin{lemma}\label{p_*=p_!}
Let $X$ be a compact real analytic manifold acted on by $U$.
Let $p:U\backslash X\to\on{pt}$ be the projection map.
We have $p_{*,\on{ren}}\is p_![-d_{U}+d_{U_c}]$.
\end{lemma}
\begin{proof}
Since the projection factors as 
$p:U\backslash X\to U\backslash\on{pt}\to\on{pt}$, where the first map is proper, we can assume $X=\on{pt}$. 
Let $\sigma:\on{pt}\to U\backslash\on{pt}$ be the projection.
We have $\sigma_{*,\on{ren}}\bC_{\on{pt}}\is\sigma_*\bC_{\on{pt}}$ (as $\sigma$ is schematic) and
$\sigma_*\bC_{\on{pt}}\is\sigma_!\bC_{\on{pt}}[-d_U+d_{U_c}]$,
and it follows that 
\[p_![-d_{L}+d_{L_c}](\sigma_*\bC_{\on{pt}})\is p_![-d_{U}+d_{U_c}]\circ\sigma_![-d_{U}+d_{U_c}](\bC_{\on{pt}})\is\bC_{\on{pt}}\]
\[p_{*,\on{ren}}(\sigma_*\bC_{\on{pt}})\is p_{*,\on{ren}}\circ\sigma_{*,\on{ren}}(\bC_{\on{pt}})\is\bC_{\on{pt}}.\]
Since the category $\on{Shv}(U\backslash\on{pt})$ is compactly generated by
$\sigma_*\bC_{\on{pt}}$, we conclude that $p_![-d_{U}+d_{U_c}]\is p_{*,\on{ren}}$.

\end{proof}
\begin{remark}\label{twisting}
To see that ``twisting'' $(-)\otimes\ell_{H\backslash X}$ 
in~\eqref{arrow} is necessary 
(in contrast to the complex analytic setting), let us consider the case 
when $H=U$ is trivial and $X$ is a compact non-oriented real analytic manifold.
Then $\on{Av}^{H/U}_*=\on{Av}^{H/U}_!=\on{Id}$, 
$\ell_{H\backslash X}=\on{or}_X$ and 
and the arrow $(-)\otimes\on{or}_X\to\on{Ps-Id}_X[d_X]$ 
in~\eqref{arrow}
(which is an isomorphism) comes from the map between the corresponding kernels
\[(\Delta_X)_*(\omega_X\otimes\on{or}_X)\is
(\Delta_X)_*(\bC_X)[d_X]\is(\Delta_X)_!(\bC_X)[d_X].\]
Here $\Delta_X:X\to X\times X$ is the diagonal closed embedding.
\end{remark}

\quash{
\begin{lemma}\label{lemma 2}
Any object in the essential image of the forgetful functor 
\[\on{oblv}_{H^-/U}:\on{Shv}(H^-\backslash X)\ra \on{Shv}(U\backslash X)\] is transversal with respect to $U\backslash X\ra H^+\backslash X$. The same is trure if we replace $H^+$ by $H^-$.
\end{lemma}
\begin{proof}
This follows from the fact that 
$H^+$ and $H^-$-orbtis on $X$ are transversal.
\end{proof}}

\section{Main results}
\subsection{}
We now fix a Matsuki datum $(G,X,H^+,H^-)$ as in the Section \ref{Matsuki duality}.
Let $U_c$ be the maximal compact subgroup of 
$U=H^+\cap H^-$.
The following theorems are the main results of the paper.
\begin{thm}\label{main 1}
There is a canonical isomorphism of functors $\on{Shv}(H^+\backslash X)\ra \on{Shv}(H^+\backslash X)$:
\[\on{Ps-Id}_{H^+\backslash X}\is(-\otimes\ell_{H^+\backslash X})\circ\Upsilon_!^{H^-\ra H^+}\circ\Upsilon^{H^+\ra H^-}_![-d_X+d_{U_c}].\]
\end{thm}
\begin{proof}
By Theorem \ref{Matsuki}, $\Upsilon_!^{H^-\ra H^+}
=\on{Av}_!^{H^+/U}\circ\on{oblv}_{H^-/U}$ 
is an equivalence with inverse given by 
\[\Upsilon^{H^+\to H^-}_*=\on{Av}_*^{H^-/U}\circ\on{oblv}_{H^+/U}.\]
Thus it suffices to show that there is an isomorphism 
\[
\on{Av}_!^{H^+/U}\circ\on{oblv}_{H^-/U}\is
(-\otimes\ell_{H^+\backslash X})\circ\on{Ps-Id}_{H^+\backslash X}\circ
\on{Av}_*^{H^+/U}\circ\on{oblv}_{H^-/U}[d_X-d_{U_c}].\]
The later follows from Lemma \ref{lemma 1} and Example \ref{example transversal}.

\end{proof}

Let $\on{Se}_{H^+\backslash X}:\on{Shv}(H^+\backslash X)\to\on{Shv}(H^+\backslash X)$
be the Serre functor for $\on{Shv}(H^+\backslash X)$.
\begin{thm}\label{main 2}
There is a canonical isomorphism of functors $\on{Shv}(H^+\backslash X)\ra \on{Shv}(H^+\backslash X)$:
\[\on{Se}_{H^+\backslash X}\is\Upsilon_*^{H^-\ra H^+}\circ\Upsilon^{H^+\ra H^-}_*\circ(-\otimes\ell_{H^+\backslash X})[d_X-d_{U_c}].\]
\end{thm}
\begin{proof}
Since the stack $H^+\backslash X$ has only finitely many isomorphism 
class of points,
according to \cite[Example 3.3.8 and Corollary 3.4.7]{AGKRRV2}, the Serre functor $\on{Se}_{H^+\backslash X}$ and the Drinfeld-Gaitsgory functor $\on{Ps-Id}_{H^+\backslash X}$ are inverse to each other.
Now the theorem follows immediately from Theorem \ref{main 1} and Theorem \ref{Matsuki},
and the fact $\ell_{H^+\backslash X}\otimes\ell_{H^+\backslash X}\is\bC_{H^+\backslash X}$.
\end{proof}
\begin{remark}\label{modification}
In the case when both $X$ and $U$ and the $U$-action on $X$ are complex analytic 
or 
$X$ is oriented and 
$U$ is connected, the orientation sheaf $\on{or}_{U\backslash X}\is\bC_{U\backslash X}$ is trivial. Indeed, the first case is a well-known fact 
and the second case follows from the fact that 
$\on{or}_\sigma\is\bC_X$ where $\sigma:X\to U\backslash X$
(see, e.g., \cite[Lemma 7.5.3]{BL}), and hence 
$\sigma^*\on{or}_{U\backslash X}\is\on{or}_\sigma\otimes\on{or}_X\is\on{or}_X\is\bC_X$ (as $X$ is oriented). Since $\sigma$ is smooth with connected fiber it implies 
$\on{or}_{U\backslash X}\is\bC_{U\backslash X}$.
Thus in the above situation, the
functor $(-)\otimes\ell_{H^+}$ is isomorphic to the identity functor
$(-)\otimes\ell_{H^+\backslash X}\is\on{Id}_{H^+\backslash X}$
and the formulas above becomes just
\[\on{Ps-Id}_{H^+\backslash X}\is\Upsilon_!^{H^-\ra H^+}\circ\Upsilon^{H^+\ra H^-}_![-d_X+d_{U_c}]\]
\[\on{Se}_{H^+\backslash X}\is\Upsilon_*^{H^-\ra H^+}\circ\Upsilon^{H^+\ra H^-}_*[d_X-d_{U_c}].\]

\end{remark}

\subsection{The case of twisted sheaves}\label{twisted}
We extend Theorem \ref{main 1} and Theorem \ref{main 2} to the setting of twisted sheaves.
For simplicity, we consider the situation when $X$ is the flag manifold of
a complex reductive group $G$.
Write $\calS= H^\pm\backslash X$.
Then for any character $\lambda$ of the Lie algebra of the universal Cartan group of $G$, we have the dg-category 
$\on{Shv}(\calS)^{\on{constr}}_\lambda$ 
of 
$\lambda$-twisted bounded $\bC$-constructible sheaves on 
$\calS$ (see, e.g., \cite[Section 6]{MUV})
and its ind-completion $\on{Shv}(\calS)_\lambda$.

We have similar six functor formalism in the twisted setting 
and 
any object 
$\mQ\in\on{Shv}(\calS\times\calS)_{-\lambda,\lambda}$ defines a functor
\[F_\mQ:\on{Shv}(\calS)\to\on{Shv}(\calS),\ \ \mF\to (p_2)_{*,\on{ren}}(\mQ\otimes^!(p_1)^!(\mF)).\] 
Consider the diagonal embedding 
$\Delta_\calS:\calS\to\calS\times\calS$.
The Drinfeld-Gaitsgory functor for $\on{Shv}(\calS)_\lambda$ is 
given by 
\[\on{Ps-Id}_{\calS}:=F_\mQ:\on{Shv}(\calS)_\lambda\to\on{Shv}(\calS)_\lambda\] with 
\[\mQ=(\Delta_{\calS})_*(\bC_\calS)\in\on{Shv}(\calS\times\calS)_{-\lambda,\lambda}\]
where we note that the pullback of the $(-\lambda,\lambda)$-twisting 
along the diagonal map is canonically trivial and hence the object above 
is well-defined.
Let $\omega_{U\backslash X}\in\on{Shv}(U\backslash X)_\lambda$ be the 
dualizing complex and 
let $\ell_{H^+\backslash X}$ to be the unique rank one $\lambda$-twisted local system on 
 $H^+\backslash X$ such that 
its pullback to $U\backslash X$ is isomorphic to 
$\on{or}_{U\backslash X}=\calH^{-d_{U\backslash X}}(\omega_{U\backslash X})$. 

Now all the discussion in the previous section work in the 
twisted setting and we have 
\begin{thm}\label{main 1 twisted}
There is a canonical isomorphism of functors $\on{Shv}(H^+\backslash X)_\lambda\ra \on{Shv}(H^+\backslash X)_\lambda$:
\[\on{Ps-Id}_{H^+\backslash X}\is(-\otimes\ell_{H^+\backslash X})\circ\Upsilon_!^{H^-\ra H^+}\circ\Upsilon^{H^+\ra H^-}_![-d_X+d_{U_c}]\]
\end{thm}
Let
$\on{Se}_{H^+\backslash X}$ be the  
the Serre functors for $\on{Shv}(H^+\backslash X)_\lambda$. 
\begin{thm}\label{main 2 twisted}
There is a canonical isomorphism of functor $\on{Shv}(H^+\backslash X)_\lambda\ra \on{Shv}(H^+\backslash X)_\lambda$:
\[\on{Se}_{H^+\backslash X}\is\Upsilon_*^{H^-\ra H^+}\circ\Upsilon^{H^+\ra H^-}_*\circ(-\otimes\ell_{H^+\backslash X})[d_X-d_{U_c}].\].
\end{thm}

\end{document}